\documentclass[a4paper,notitlepage]{article}%
\usepackage{amsfonts}
\usepackage{hyperref}
\usepackage{amsmath}
\usepackage{amssymb}
\usepackage{graphicx}%
\providecommand{\U}[1]{\protect\rule{.1in}{.1in}}
\newtheorem{theorem}{Theorem}

\newtheorem{remark}[theorem]{Remark}

\newenvironment{proof}[1][Proof]{\noindent\textbf{#1.} }{\ \rule{0.5em}{0.5em}}
\begin{document}

\title{Lyapunov stability of an SIRS epidemic model with varying population}
\author{Alberto d'Onofrio$^{\text{a}}$, Piero Manfredi$^{\text{b}}$, Ernesto
Salinelli$^{\text{c}}$\thanks{Corresponding author}\\$^{\text{a}}$International Prevention Research Institute, \\15 Chemin du Saquin, BAT.G, \\69130 Ecully (Lyon), France\\$^{\text{b}}$Department of Economics and Management,\\University of Pisa,\\Via C. Ridolfi, 10, 56124 Pisa, Italy\\$^{\text{c}}$Department of Translational Medicine, \\University of Piemonte Orientale, \\Via Solaroli 17, 28100 Novara, Italy}
\date{}
\maketitle

\begin{abstract}
In this paper we consider an SIRS epidemic model under a general assumption of
density-dependent mortality. We prove the global stability of the disease-free
equilibrium and propose a Lyapunov function that allows to demonstrate the
global stability of the (unique) endemic state under broad conditions.


\end{abstract}

\section{Introduction}

Since the seminal work initiated a century ago by Kermack and McKendrik (see
\cite{KeKe27}, \cite{CA08}), Mathematical Epidemiology has undertaken an
extraordinary development to the point that mathematical models nowadays
represent a key support of public policies aimed to control infectious
diseases. All this is based on refinements of a few basic deterministic models
with a simple structure, where the population is divided into states or
\textquotedblleft compartments\textquotedblright\ representing the status with
respect to the infection such as e.g., susceptible, infectious, and immune, as
in classical SIR and SIRS models (see \cite{CA08}), which individuals can
visit according to simple transition rules. These simple structures can be
generalized allowing to (i) model any type of infectious diseases, ranging
from vaccine preventable to vertically or sexually transmitted and to
vector-borne \cite{Anderson}, to (ii) include any type of intermediate
determinants of epidemiological outputs such as e.g., the role of individual's
age in transmission, geographic structures, population dynamics, and control
variables \cite{Anderson}, such as vaccination or treatments, up to the
psychological dimension of human behavior \cite{MAdO13}, and to (iii) include
different kinds of nonlinearities in social contacts and transmission
processes \cite{CA08}.\newline The basic models of infection spread are
deterministic, continuous time, and expressed by a system of nonlinear
ordinary differential equation (ODEs). The main dynamic features of these
models are the existence of a disease-free equilibrium and, provided an
appropriate parameter representing the reproduction of infection is above a
threshold, of an endemic equilibrium where the infection persists. A key issue
regards the (local and/or global) stability of these equilibria.\newline The
involved mathematics follow two alternative approaches. The first one is based
on Poincar\'{e}-Bendixon theory for planar systems and its recent
multidimensional extension (see \cite{LiMu96} and \cite{BudOLa08} for an
application). The second uses the classical Lyapunov direct method, still
widely applied (see \cite{Ko06}, \cite{KoMa04}, \cite{KoWa02},
\cite{LaOmKi2011}, \cite{LaOmSeBe15}, \cite{LiYaXiLi2016}, \cite{LiZhLiCh17},
\cite{OR10} and \cite{Va11}).

\noindent In \cite{OR10}, the following SIRS model was proposed%
\begin{equation}
\left\{
\begin{array}
[c]{l}%
X^{\prime}=b(N-pY)-\mu X-\beta X\dfrac{Y}{N}+\alpha Z\\
Y^{\prime}=bpY+\beta X\dfrac{Y}{N}-\left(  \mu+\nu+\delta\right)  Y\\
Z^{\prime}=\nu Y-\left(  \alpha+\mu\right)  Z
\end{array}
\right. \label{SIRS_R}%
\end{equation}
where: $b>0$ and $\mu>0$ are the birth and death rates, $\beta>0$ is the
transmission rate, $\nu>0$ the rate of recovery from infection, $\delta>0$ the
disease-specific mortality rate, $p\in\left(  0,1\right)  $ the fraction of
vertically infected newborn and $\alpha>0$ the rate of return to
susceptibility by loss of immunity. The stability analysis of the endemic
equilibrium was performed by Lyapunov direct method under the very special
assumption of constant population size $N$. However, it is easy to verify that
$N=X+Y+Z$ is constant under the highly special condition on model parameters%
\[
\beta\left(  b-\mu\right)  \left(  \alpha+\mu+\nu\right)  =\delta\left(
\alpha+\mu\right)  \left(  \left(  pb+\beta\right)  -\left(  \mu+\nu
+\delta\right)  \right)
\]
which makes this case totally uninteresting as missing any non trivial
dynamics.\newline The aim of this paper is therefore to re-analyze the
stability of steady states of model (\ref{SIRS_R}) in non trivial conditions
by appropriately reformulating the dynamics of the population under a general
assumption of density-dependent mortality. We demonstrate the global stability
of the disease-free equilibrium and propose a Lyapunov function that allows to
demonstrate the stability of the (unique) endemic state under broad conditions.

\section{Equilibria and their stability}

Following \cite{GaHe}, \cite{d'OMaSa08}, we assume that $\mu=\mu\left(
N\right)  $ is strictly increasing, with $b>\mu\left(  0\right)  $ and
$b<\mu\left(  +\infty\right)  $. The resulting population dynamics is%
\begin{equation}
N^{\prime}=\left(  b-\mu\left(  N\right)  \right)  N-\delta
Y.\label{dynamic_population_true}%
\end{equation}
Note that, since the equation $y^{\prime}=\left(  b-\mu\left(  y\right)
\right)  y$ admits a unique globally asymptotically stable (GAS) equilibrium
$y^{\ast}$, our variant to model (\ref{SIRS_R}) can be studied on the
positively invariant set%
\[
\mathcal{D}=\left\{  \left(  X,Y,Z,N\right)  :X\geq0,Y\geq0,Z\geq0,X+Y+Z=N\leq
N^{\ast}\right\}  .
\]
Since $N$ is not constant, it is convenient to pass to fractions $S=X/N$,
$I=X/N$ and $R=Z/N$, obtaining the same epidemiological structure of system
(\ref{SIRS_R}):%
\begin{equation}
\left\{
\begin{array}
[c]{l}%
S^{\prime}=b(1-S-pI)-\left(  \beta-\delta\right)  SI+\alpha R\\
I^{\prime}=\left(  \beta S-\left(  \left(  1-p\right)  b+\nu+\delta\right)
+\delta I\right)  I\\
R^{\prime}=\nu I-\left(  b+\alpha-\delta I\right)  R
\end{array}
\right.  \label{SIRS_fra}%
\end{equation}
completed by the following equation for the population:
\[
N^{\prime}=N\left(  b-\mu\left(  N\right)  \right)  -\delta I.
\]
It is simple to verify that the region%
\begin{equation}
\mathcal{D}^{\text{fra}}=\left\{  \left(  S,I,R\right)  :S\geq0,I\geq
0,R\geq0,~S+I+R=1\right\}  \label{region}%
\end{equation}
is positive invariant and attractive. Moreover, model (\ref{SIRS_fra}) always
admits the disease-free equilibrium $E_{0}=\left(  1,0,0\right)
\in\mathcal{D}^{\text{fra}}$. By setting%
\[
\ \ \ \ \ \ \ \ \ \ \ \ R_{0}=\dfrac{\beta}{\gamma},\ \ \ \ \gamma=\left(
1-p\right)  b+\nu+\delta
\]
it results that $E_{0}$ is the unique equilibrium when $R_{0}\leq1$, where
$R_{0}$ represents the appropriate reproduction number for system
(\ref{SIRS_fra}). Indeed, following \cite{BuDr90}, by conditions $S^{\prime
}=0$, $I^{\prime}=0$ and $S=1-I-R$, for $S>0$ and $I>0$, one easily obtain the
equality%
\[
\gamma\left(  1-R_{0}\right)  S+b\left(  1-S-pI\right)  +\left(  \alpha+\beta
S\right)  R=0
\]
that shows that no solutions can exist when $R_{0}\leq1$.\newline Next, we
prove that, if $R_{0}>1$, a unique endemic equilibrium $E_{1}=\left(
S_{e},I_{e},R_{e}\right)  \in\mathcal{D}^{\text{fra}}$ for (\ref{SIRS_fra})
exists. Since $S+I+R=1$, we can consider the reduced system
\begin{equation}
\left\{
\begin{array}
[c]{l}%
I^{\prime}=\left(  \beta\left(  1-I-R\right)  -\gamma+\delta I\right)  I\\
R^{\prime}=\nu I-\delta\left(  \rho-I\right)  R
\end{array}
\right.  \label{reduced}%
\end{equation}
where $\rho=\left(  b+\alpha\right)  /\delta$, on the invariant set
\[
\overline{\mathcal{D}}_{0}^{\text{fra}}=\left\{  \left(  I,R\right)
:I\geq0,R\geq0,~I+R\leq1\right\}  \backslash\left\{  \left(  0,0\right)
\right\}
\]
and show that (\ref{reduced}) has a unique positive solution $\widehat{E}%
_{1}=\left(  I_{e},R_{e}\right)  \in\overline{\mathcal{D}}_{0}^{\text{fra}}$
when $R_{0}>1$ (equivalent to $\beta>\gamma$). Note preliminarily that
condition $R_{0}>1$ implies $\beta>\delta$ and it can occur $R^{\prime}=0$
only if $I<\rho$. By solving $I^{\prime}=R^{\prime}=0$, we obtain the
quadratic equation in $I$%
\[
P\left(  I\right)  =-\delta\left(  \beta-\delta\right)  I^{2}+\left(
\delta\rho\left(  \beta-\delta\right)  +\left(  \beta-\gamma\right)
\delta+\beta\nu\right)  I-\delta\rho\left(  \beta-\gamma\right)  =0.
\]
Since ($P\left(  0\right)  <0$) and%
\[
\lim\limits_{I\rightarrow\rho^{-}}P\left(  I\right)  =\beta\nu\rho>0
\]
there is only one solution of $P\left(  I\right)  =0$ smaller than $\rho$.
Therefore, if $\rho\leq1$ the proof immediately follows, while if $\rho>1$
note that $P\left(  1\right)  =\delta\left(  \rho-1\right)  \left(
\gamma-\delta\right)  +\beta\nu>0$, and again the claim follows.\newline We
first obtain the global stability of the disease-free equilibrium by adopting
the Lyapunov function $I(t)$.

\begin{theorem}
The disease-free equilibrium $E_{0}\in\mathcal{D}^{\text{fra}}$ is GAS in
$\mathcal{D}^{\text{fra}}$ if and only if $R_{0}\leq1$, and unstable for
$R_{0}>1$.
\end{theorem}

\begin{proof}
By linearization it is easy to see that $E_{0}$ is locally asymptotically
stable when $R_{0}<1$, and unstable when $R_{0}>1$. We assume in the following
$R_{0}\leq1$ and we show that $L_{\text{DFE}}\left(  t\right)  =I\left(
t\right)  $ is a Lyapunov function. In fact, as $S=1-I-R$, we can write%
\[
I^{\prime}=\left(  \gamma\left(  R_{0}-1\right)  -\beta R-\left(  \beta
-\delta\right)  I\right)  I.
\]
Then, by $R_{0}\leq1$ and $\beta>\delta$, we immediately obtain $I^{\prime
}\leq0$. If $\delta\geq\beta$, for $R_{0}<1$, as $\gamma\left(  R_{0}%
-1\right)  =\beta-\gamma$, it follows
\[
I^{\prime}=\left(  -\beta R-\left(  \delta-\beta\right)  (1-I)-\left(  \left(
1-p\right)  b+\nu\right)  \right)  I\leq0.
\]
Since the DFE is the only positively invariant subset of $\left\{  \left(
S,I,R\right)  :I^{\prime}=0\right\}  $, by LaSalle Invariance Principle, we
conclude that $E_{0}$ is GAS for $R_{0}\leq1$.
\end{proof}

\noindent We show now that, under suitable assumptions, there exists a
Lyapunov function for the endemic equilibrium $\widehat{E}_{1}$ of system
(\ref{reduced}).

\begin{theorem}
If $R_{0}>1$ and $\beta\leq\gamma+\dfrac{\rho\left(  \gamma-\delta\right)
}{1-\rho}$, the unique endemic equilibrium $\widehat{E}_{1}$ of system
(\ref{reduced}) is GAS on $\overline{\mathcal{D}}_{0}^{\text{fra}}$.
\end{theorem}

\begin{proof}
Since, by definition, at $\widehat{E}_{1}$ it holds:%
\begin{equation}
\left\{
\begin{array}
[c]{l}%
\beta\left(  1-I_{e}-R_{e}\right)  -\gamma+\delta I_{e}=0\\
\nu I_{e}-\delta\left(  \rho-I_{e}\right)  R_{e}=0
\end{array}
\right.  \label{EE}%
\end{equation}
and $IR-I_{e}R_{e}=I\left(  R-R_{e}\right)  +R_{e}\left(  I-I_{e}\right)  $,
we can rewrite system (\ref{reduced}) as:%
\begin{equation}
\left\{
\begin{array}
[c]{l}%
I^{\prime}=-\left(  \left(  \beta-\delta\right)  \left(  I-I_{e}\right)
-\beta\left(  R-R_{e}\right)  \right)  I\\
R^{\prime}=\left(  \nu+\delta R_{e}\right)  \left(  I-I_{e}\right)
-\delta\left(  \rho-I\right)  \left(  R-R_{e}\right)
\end{array}
\right.  \nonumber
\end{equation}
Let us now consider the positive functions on $\overline{\mathcal{D}}%
_{0}^{\text{fra}}\backslash\left\{  \left(  I_{e},R_{e}\right)  \right\}  $
\[
L_{1}\left(  t\right)  =I-I_{e}-I_{e}\ln I;\ \ \ \ \ L_{2}\left(  t\right)
=\dfrac{\beta}{2\left(  \nu+\delta R_{e}\right)  }\left(  R-R_{e}\right)
^{2}.
\]
Along the solutions of (\ref{reduced}) we have%
\begin{align*}
\dfrac{dL_{1}\left(  t\right)  }{dt} &  =\dfrac{I-I_{e}}{I}I^{\prime}=-\left(
\beta-\delta\right)  \left(  I-I_{e}\right)  ^{2}-\beta\left(  R-R_{e}\right)
\left(  I-I_{e}\right)  \\
\dfrac{dL_{2}\left(  t\right)  }{dt} &  =\dfrac{\beta}{2}\left(
I-I_{e}\right)  \left(  R-R_{e}\right)  -\delta\left(  \rho-I\right)  \left(
R-R_{e}\right)  ^{2}%
\end{align*}
Therefore, function $L_{\text{EE}}=L_{1}+L_{2}$ is positive for $\left(
I,R\right)  \neq\left(  I_{e},R_{e}\right)  $ and%
\[
L_{\text{EE}}^{\prime}\left(  t\right)  =-\left(  \beta-\delta\right)  \left(
I-I_{e}\right)  ^{2}-\delta\left(  \rho-I\right)  \left(  R-R_{e}\right)
^{2}.
\]
Note that if $\rho\geq1$, then $L_{\text{EE}}$ is a Lyapunov function and the
endemic equilibrium is GAS. In passing, we also note that (see the second
equation in (\ref{EE}) ) it holds $I_{e}<\rho$, which means that
$L_{\text{EE}}$ acts as a local Lyapunov function for system (\ref{reduced}%
).\newline Consider now case $\rho<1$. Observe that the set $\left\{  \left(
I,R\right)  \in\overline{\mathcal{D}}_{0}^{\text{fra}}:I^{\prime}=0\right\}  $
is a straight line that intersects the line $I=0$ at a point $I_{u}$
fulfilling%
\[
0<I_{u}=\dfrac{\beta-\gamma}{\beta-\delta}=\dfrac{\beta-\left(  \left(
1-p\right)  b+\nu+\delta\right)  }{\beta-\delta}=1-\dfrac{\left(  1-p\right)
b+\nu}{\beta-\delta}<1.
\]
If $I_{u}\leq\rho$, the set $\Omega=\left\{  \left(  I,R\right)  \in
\overline{\mathcal{D}}_{0}^{\text{fra}}:I<\rho\right\}  $ is positively
invariant. In fact, it is easy to verify that $I_{u}\leq\rho$ if and only if%
\[
\beta-\gamma-\rho\left(  \beta-\delta\right)  \leq0.
\]
This implies that on $\left\{  \left(  I,R\right)  \in\overline{\mathcal{D}%
}_{0}^{\text{fra}}:I=\rho\right\}  $ we have $I^{\prime}\leq0$. Furthermore,
$\Omega$ is attractive as%
\[
I^{\prime}\leq\beta-\gamma-\left(  \beta-\delta\right)  I<\beta-\gamma
-\rho\left(  \beta-\delta\right)
\]
for each $I\in\left[  \rho,I_{u}\right]  $. Since $L_{\text{EE}}$ is a
Lyapunov function on $\Omega$, the endemic equilibrium $\widehat{E}_{1}$ is
GAS when $I_{u}\leq\rho$.
\end{proof}

\noindent The result just obtained straightforwardly extends to the endemic
equilibrium $E_{1}$ of system (\ref{SIRS_fra}).

\begin{remark}
It is possible to verify that the case $I_{u}>\rho$ in the previous proof is
far from trivial. In fact, $I_{u}>\rho$ if and only if
\[
\beta-\gamma-\rho\left(  \beta-\delta\right)  >0.
\]
Indeed, with simple manipulations, we obtain the equivalent condition%
\[
\beta>\gamma+\dfrac{\rho\left(  \gamma-\delta\right)  }{1-\rho}.
\]
showing that this particular case deals with situations where the interplay
between demographic and epidemiological parameters favour a relatively high
infection transmission.
\end{remark}

\begin{remark}
The present characterization of the stability of equilibria allows clear
conclusions about the effects of endemicity on the dynamics of the population
\[
N^{\prime}=\left(  b-\mu\left(  N\right)  -\delta I\right)  N.
\]
Indeed, if the endemic state is GAS, then
\[
N^{\prime}\rightarrow(b-\delta I_{e}-\mu(N))N
\]
which implies that the disease will bring the extinction of the population
under the condition:
\[
b\leq\delta I_{e}+\mu(0).
\]
Conversely, if the previous condition is not met, then the disease and the
population will reach the equilibrium state $N_{e}=\mu^{-1}(b-\delta I_{e})$
where the persistent presence of the disease will regulate the population size
\cite{BuDr90}, \cite{GaHe}, \cite{MeHe92} \cite{d'OMaSa08}.
\end{remark}

\noindent\textbf{Acknowledgments. }The financial support of Universit\`{a} del
Piemonte Orientale is acknowledged by the authors.


\begin{thebibliography}{99}                                                                                               %
\bibitem {Anderson}R.M. Anderson, R.M. May, 1996, Infectious diseases of
humans, \emph{Oxford Univ Press}, Oxford.

\bibitem {BeCa86}E. Beretta, V. Capasso, 1986. On the general structure of
epidemic systems. Global asymptotic stability. \emph{Comp. \& Maths. with
Appls.} 12A, 6, 677-694.

\bibitem {BudOLa08}B. Buonomo, A. d'Onofrio, D. Lacitignola, 2008, Global
stability of an SIR epidemic model with information dependent vaccination,
\emph{Mathematical Biosciences} 216, 9-16.

\bibitem {BuDr90}S. Busenberg, P. van den Driessche, 1990, Analysis of a
disease transmission model in a population with varying size, \emph{J. Math.
Biology}, 28, 257-270.

\bibitem {CA08}V. Capasso, Mathematical Structures of Epidemic Systems,
Lectures Notes in Biomathematics 97, II ed., 2008, Springer-Verlag.

\bibitem {donofrio1}A. d'Onofrio, P. Manfredi, E. Salinelli, 2007. Vaccinating
behaviour, information, and the dynamics of SIR vaccine preventable diseases,
\emph{Theoretical Population Biology} 71, 301-317.

\bibitem {d'OMaSa08}A. d'Onofrio, P. Manfredi, E. Salinelli, 2008, Fatal SIR
diseases and rational exemption to vaccination, \emph{Mathematical Medicine
and Biology} 25, 337-357.

\bibitem {GaHe}L.Q. Gao., H.W. Hethcote, 1992, Disease transmission models
with density-dependent demographics, \emph{J. Math. Biology}, 30, 717-731.

\bibitem {KeKe27}W.O. Kermack, A.G. McKendrick, A contribution to the
mathematical theory of epidemics, \emph{Proc. R. Soc. Lond. Ser. A Math. Phys.
Eng. Sci.}, 115 (1927), 700--721.

\bibitem {Ko06}A. Korobeinikov, Lyapunov functions and global stability for
SIR and SIRS epidemiological models with non-linear transmission, \emph{Bull.
Math. Biol.}, 68 (2006), 615--626.

\bibitem {KoMa04}A. Korobeinikov, P.K. Maini, A Lyapunov function and global
properties for SIR and SEIR epidemiological models with nonlinear incidence,
\emph{Math. Biosci. Eng.}, 1 (2004), 57--60.

\bibitem {KoWa02}A. Korobeinikov, G.C. Wake, Lyapunov functions and global
stability for SIR, SIRS, and SIS epidemiological models, \emph{Appl. Math.
Lett.}, 15 (2002), 955--960.

\bibitem {LaOmSeBe15}A. Lahrouz, L. Omari, A. Settati, A. Belmaati, 2015,
Comparison of deterministic and stochastic SIRS epidemic model with saturating
incidence and immigration, \emph{Arabian Journal of Mathematics} 4, 101-116

\bibitem {LaOmKi2011}A. Lahrouz, L. Omari, D. Kiouach, 2011, Global analysis
of a deterministic and stochastic nonlinear SIRS epidemic model,
\emph{Nonlinear Analysis: Modelling and Control} 16, 1, 59-76.

\bibitem {LiYaXiLi2016}J. Li, Y. Yang, Y. Xiao, S. Liu, 2016, A class of
Lyapunov functions and the global stability of some epidemic models with
nonlinear incidence, \emph{J. Applied Analysis and Computation} 6, 1, 38-46.

\bibitem {LiMu96}M.Y. Li, J.S. Muldowney, A geometric approach to
global-stability problems, SIAM J. Math. Anal. 27 (4), 1070-1083.

\bibitem {LiZhLiCh17}T. Li, F. Zhang, H. Liu, Y. Chen, 2017, Threshold
dynamics of an SIRS model with nonlinear incidence rate and transfer from
infectious to susceptible, \emph{Applied Methematics Letters} 70, 52-57.

\bibitem {MAdO13}P. Manfredi, A. d'Onofrio (Eds.), Modeling the Interplay
Between Human Behavior and the Spread of Infectious Diseases, 2013, Springer.

\bibitem {MeHe92}J. Mena-Lorca, H.W. Hethcote, Dynamic models of infectious
diseases as regulator of population sizes, \emph{J. Math. Biol.}, 30 (1992), 693--716.

\bibitem {OR10}S.M. O'Regan, T.C. Kelly, A. Korobeinikov, M.J.A. O'Callaghan,
A.V. Pokrovskii, 2010, Lyapunov functions for SIR and SIRS epidemic models,
Applied Mathematics Letters 23, 446-448.

\bibitem {RO11}R. Ross, The prevention of malaria. London: John Murray; 1911.

\bibitem {Va11}C. Vargas-De-Le\`{o}n, 2011. On the global stability of SIS,
SIR and SIRS epidemic models with standard incidence, \emph{Chaos, Solitons \&
Fractals} 44, 1106-1110.
\end{thebibliography}
\end{document}